\documentclass[a4paper,12pt,leqno]{amsart}
\usepackage[english]{babel}
\usepackage{amsmath,amsthm,amssymb}
\usepackage{xcolor,graphicx,delarray}
\usepackage{comment}
\usepackage[T1]{fontenc}
\usepackage{floatflt,tikz,pgf,pgfplots}
\usepackage[toc,page]{appendix}
\usepackage{cancel}
\numberwithin{equation}{section}
\usepackage{subcaption}
\newcommand{\la}{\lambda}
\newcommand{\al}{\alpha}

\newcommand{\notequal}{\neq}
\newcommand{\be}{\begin{equation}}
\newcommand{\bea}{\begin{eqnarray}}
\newcommand{\beas}{\begin{eqnarray*}}
\newcommand{\ee}{\end{equation}}
\newcommand{\eea}{\end{eqnarray}}
\newcommand{\eeas}{\end{eqnarray*}}

\newcommand{\R}{\mathbb{R}}

\newcommand{\Sd}{\mathbb{S}^2}

\newcommand{\dist}{\rm dist}

\renewcommand{\L}{{\mathcal L}(S)}
\renewcommand{\H}{{\mathcal H}(S)}

\newcommand{\E}{\mathcal E}
\newcommand{\F}{\mathcal F}

\DeclareMathOperator{\rank}{rank}

\newtheorem{theorem}{Theorem}[section]
\newtheorem{Lemma}[theorem]{Lemma}
\newtheorem{proposition}[theorem]{Proposition}

\newtheorem{definition}[theorem]{Definition}
\newtheorem{remark}[theorem]{Remark}

\newlength{\temparglen}

\newcommand{\pagenlarge}[1]{
\setlength{\temparglen}{#1\topmargin}
\addtolength{\textheight}{2\temparglen}
\addtolength{\topmargin}{-\temparglen}
\setlength{\temparglen}{#1\oddsidemargin}
\addtolength{\textwidth }{2\temparglen}
\addtolength{\oddsidemargin }{-\temparglen}
\addtolength{\evensidemargin }{-\temparglen}
}
\pagenlarge{1}
\title[Stability under lamination and polycrystalline effective conductivities]{Stability under lamination and polycrystalline effective conductivity}

\author[N. Albin]
{N. Albin}
\address[Nathan Albin]{Department of Mathematics, Kansas State University, 138 Cardwell Hall, 1228 N.~17th Street, Manhattan, KS 66506, USA}
\email{albin@k-state.edu}
\author[V. Nesi]
{V. Nesi}
\address[Vincenzo Nesi]{Dipartimento di Matematica, Sapienza, Universit\`a di Roma, P.le A. Moro, 00100, Rome, Italy}
\email{vincenzo.nesi@uniroma1.it}

\author[M. Palombaro]
{M. Palombaro}
\address[Mariapia Palombaro]{DISIM, Universit\`a dell'Aquila, Via Vetoio, 67100 L'Aquila, Italy}
\email{mariapia.palombaro@univaq.it}
\date{\today}
\pgfplotsset{compat=1.18} 
\begin{document}
\begin{abstract}
We prove the stability under lamination of a set of real, symmetric 3$\times$3 matrices that can be viewed as a subset of the effective conductivities of a polycrystal.  Constructed in the companion paper \cite{ANP}, such set in combination with several previous constructions \cite{ACLM}
provides the best inner bound known so far on the  $G$-closure of a three dimensional polycrystal.

\par\medskip
{\noindent\textbf{Keywords:}
Effective conductivity, differential inclusions, laminates, stability under lamination.
}
\par{\noindent\textbf{MSC2020:} 
35B27, 49J45}
\end{abstract}

\maketitle

\thispagestyle{empty}

\section{Introduction}
\noindent 
In the recent paper \cite{ANP} the authors address the problem of finding solutions to a differential inclusion arising in the context of bounding the effective conductivity of a three-dimensional polycrystal, \cite{ACLM}, \cite{NM}, \cite{Sch}.
In the present follow-up paper, we focus on a certain property of the class of solutions 
described in \cite{ANP}. 
We start by recalling the differential inclusion and refer the reader to \cite[\S 2]{ANP} for its 
physical motivation and the exact link with the polycrystal problem.

Let $\mathbb M^{3\times 3}_{\rm sym}$ denote the set of real, symmetric $3\times 3$ matrices and let 
$ S \in\mathbb M^{3\times 3}_{\rm sym}$ be given by
\begin{equation}\label{S}
S=\left(
\begin{array}{ccc}
s_1&0&0\\
0&s_2&0\\
0&0&s_3
\end{array}
\right)
\end{equation}
subject to the constraints
\begin{equation}\label{cS}
0<s_1 < s_2< s_3,\quad s_1+s_2+s_3=1.
\end{equation}

We define the 
set $K(S)\subset \mathbb M^{3\times 3}_{\rm sym}$ as follows:
\begin{equation}\label{setK}
K(S) :=\{
\lambda R^t S R: \lambda\in\R, \, R\in SO(3)
\}.
\end{equation}
Notice that $K(S)$ is unbounded.
Set  $C=[0,1]^3$ and denote by $W_{C}^{1,2}(\mathbb R^3;\mathbb R^3)$  the space of vector fields in 
$W^{1,2}_{\rm loc}(\mathbb R^3;\mathbb R^3)$ that are $C$-periodic.  
We look for  $A\in\ \mathbb M^{3\times 3}_{\rm sym}$ such that the following differential inclusion admits solutions
\begin{equation}\label{diffin}
 \nabla u\in K(S) \, \text{ a.e.}, \quad u-Ax \in W_{C}^{1,2}(\R^3; \R^3).
\end{equation}
Solutions are understood in the approximate sense of the existence of sequences $\{u_j\}\subset W_{C,A}^{1,2}\equiv W_{C}^{1,2}(\R^3; \R^3)+ A x$, 
that are bounded in $W^{1,2}_{\text{loc}}(\R^3;\R^3)$ and such that $\{\nabla u_j\}$ is 
$L^2_{\text{loc}}$-equi-integrable and
\begin{equation*}
{\dist} (\nabla u_j, K(S))\to 0 \text{ locally in measure}. 
\end{equation*}
We denote the set of all such $A$'s as $K^{app}(S)$ and point out that  $K^{app}(S)$ is related to the so-called 2-quasiconvex hull of $K(S)$ as defined in \cite{Yan}, with $K^{app}(S)$ being a subset if it 
(see also  \cite{Zhang} for a less restrictive definition of 2-quasiconvex hull). 
By definition of $K(S)$,
$K^{app}(S)$ is invariant under conjugation by any rotation, i.e., if $A\in K^{app}(S)$ then $ R^t A R \in  K^{app}(S)$ for each $R\in SO(3)$. Therefore, it suffices to characterize the eigenvalues of the 
elements of $K^{app}(S)$, which can thus be identified with a subset of $\R^3$. 
In fact, because of its original physical motivation (see \cite[\S 2]{ANP}), we are interested in those elements of $K^{app}(S)$ whose eigenvalues lie in the interval  $[s_1,s_3]$.  
Moreover, since $K^{app}(S)$ is a cone, we may then focus on the 
specific section Tr$A=1$. 
Denoting by  $\H$ the closed hexagon in $\R^3$ of vertices at $(s_1,s_2,s_3)$ and all of its permutations, our 
goal becomes to characterise the following set

\begin{equation}
 K^*(S):= \left\{A\in K^{app}(S)\!: A\in \H
\right\}.
\end{equation}
The exact characterisation of $K^*(S)$ is currently an open problem. 
The main contribution of \cite{ANP} is an inner bound for $ K^*(S)$,  denoted by $\L$,  
which improves all the bounds that were previously known ( \cite{ACLM}, \cite{NM} and \cite{Nesi}). 
The goal of the present paper is to prove that  $\L$ enjoys the stability under lamination property (see Definition \ref{def-stability}). 
Such property would be naturally fulfilled if we knew that $\L$ coincides with 
 its 2-rank-one convex hull or with $ K^*(S)$ (see, e.g., \cite{Mnotes}, \cite{Zhang}). 
However, in the absence of such results, we need to 
prove it using the very definition of stability. 
Our proofs are performed in the space of symmetric invariants rather than that of the eigenvalues. This feature is likely to be useful in different contexts where a given set enjoys the same invariance properties as $K^{app}(S)$.
\par
Let us briefly describe the connection of $\L$ with the so-called $G$-closure of a
polycrystalline conducting material.  
In the theory of composite materials,
the set of all possible effective conductivities arising as a mixture of a single crystal with anisotropic conductivity $\sigma$,
\begin{equation}\label{sigma}
\sigma=\left(
\begin{array}{ccc}
\sigma_1&0&0\\
0&\sigma_2&0\\
0&0&\sigma_3
\end{array}
\right),\quad 0<\sigma_3\leq \sigma_2\leq \sigma_1,
\end{equation}
is known as the $G$-closure, and is commonly denoted by $G(\sigma)$, see \cite{ACLM}, \cite{NM} and \cite{masterpiece}. 
 It has been established that partial knowledge of  $\partial G(\sigma)$ can lead, in some cases,  to the knowledge of larger parts of it. An interesting example is found in \cite{MN} (for more details see pp. 204 and 645--647 of \cite{masterpiece}). The underlying idea is based upon previous work done in \cite{Fra-M} and \cite{MSL}.
 The set $\L$ described in \cite{ANP} provides a new portion of $\partial G(\sigma)$ in the case of distinct eigenvalues $\sigma_i$'s.
 We remark that explicit examples of sets which are stable under lamination but not stable under homogenization and correspond to effective properties of materials, are relatively difficult to find. A remarkable one can be found in \cite{masterpiece} (see p. 690) where, using an idea from \cite{Sv}, the author exhibits a composite consisting of seven anisotropic phases with an effective elastic tensor that cannot be realized by  laminations of the seven given phases. To the best of our knowledge, in the context of three-dimensional conductivity, no examples are yet known.
\section{Preliminary definitions and results}
The present section collects definitions and results from \cite{ANP} that will be used in the rest of the paper.
\begin{definition}\label{matrices}
We denote by  $\mathcal D\subset \mathbb M^{3\times 3}_{\rm sym}$  the set of positive definite diagonal matrices with unit trace and by 
 $\mathcal D_{\rm d}\subset\mathcal D$ its subset of matrices with distinct eigenvalues.
The eigenvalues of a matrix $M\in \mathcal D$ are denoted by lower case letters  $m_1,m_2$ and $m_3$ and ordered as follows:
\begin{equation}\label{stage}
    0<m_1\leq m_2\leq m_3.
\end{equation}
\end{definition}
Note that $S \in \mathcal D_{\rm d}$, since its eigenvalues are distinct. 
In what follows $\Sd$ denotes the set of unit vectors in $\R^3$.

\begin{definition}\label{MT}
\noindent
Let $F\in \mathcal D_{\rm d}$.    
The set
${\mathcal T^{1}(F)}$ is the set of all  $G\in \mathcal D$ such
 that
\begin{equation}\label{oldbp}
\begin{array}{cc}
\exists \,R\in SO(3), n\in\Sd, \la \in \mathbb R\backslash\{1\} :
&
R^tGR = \la\, F +(1-\la) n\otimes n. 
\end{array}\end{equation}
\end{definition}
\begin{definition}\label{A(F,G)}
Assume that $F\in \mathcal D_d$  and $G\in \mathcal D$. Set
\begin{align*}
& \al_- (F,G) :=\max\left(\frac{g_1}{f_2},\frac{g_2}{f_3}\right), 
& \al_+ (F,G):= \min\left(\frac{g_1}{f_1},\frac{g_2}{f_2},\frac{g_3}{f_3}\right) ,\\
&\beta_-(F,G):= \max \left(\frac{g_1}{f_1},\frac{g_2}{f_2},\frac{g_3}{f_3}\right),
&\beta_+(F,G):=\min \left(\frac{g_2}{f_1},\frac{g_3}{f_2}\right), \\
& A_\alpha(F,G):=[\al_- (F,G), \al_+ (F,G)],
& A_\beta(F,G):=[\beta_-(F,G), \beta_+(F,G)],\\
&A(F,G) :=\Big(A_\alpha(F,G)\cup A_\beta(F,G)\Big)\setminus\{1\} \,. 
\end{align*}
The interior and the boundary of $A(F,G)$ are denoted by $\rm{Int} A(F,G)$ and $\partial A(F,G)$ respectively, and we adopt the convention $[a,b] =\emptyset$, if $a>b$.
\end{definition}

\begin{remark}\label{rem100}
From the definition it follows that  $\alpha_+(F,G)\leq 1\leq \beta_-(F,G)$ and that  $\alpha_+(F,G)=\beta_-(F,G)=1$ if and only if $f_i=g_i$ for each $i=1,2,3$.  
Hence, the set $A(F,G)$ is always the union of two disjoint, possibly empty, bounded intervals. 
\end{remark}
\noindent
The next algebraic lemma clarifies when condition \eqref{oldbp} holds.

\begin{Lemma}\label{LemmaF} 
Let $F\in \mathcal D_d$,  $G\in \mathcal D$. 
If $G\in \mathcal T^1(F)$,  then $\la \in  A(F,G)$. Conversely, if $A(F,G) \neq \emptyset$, then $G\in  \mathcal T^1(F)$ and, in particular,  for each
$\la\in  A(F,G)$, the vector $n=(n_1,n_2,n_3)$ that satisfies \eqref{oldbp} is 
determined, not uniquely, by the following equations
\begin{equation}\label{ni}
\begin{array}{cc}
\displaystyle{ n_1^2(F,G,\la):= \frac{(g_1 -\la f_1)(g_2 -\la f_1)(g_3 -\la f_1)}{\la^2(1-\la)
(f_2 - f_1)(f_3 - f_1)}},
\\\\
\displaystyle{n_2^2(F,G,\la) := \frac{(g_1 -\la f_2)(g_2 -\la f_2)(g_3 -\la f_2)}{\la^2(1-\la)
(f_3 - f_2)(f_1 - f_2)}},
\\\\
\displaystyle{n_3^2(F,G,\la):= \frac{(g_1 -\la f_3)(g_2 -\la f_3)(g_3 -\la f_3)}{\la^2(1-\la)
(f_1 - f_3)(f_2 - f_3)} }.
\end{array}
\end{equation}
In particular $F\in \mathcal T^{1}(F)$ since 
 \begin{equation*}
 A(F,F)=\left[\alpha_-(F,F),\frac{1}{\alpha_-(F,F)}\right]\setminus\{1\} =
 \left[\max\left(\frac{s_1}{s_2}, \frac{s_2}{s_3}\right), \min\left(\frac{s_2}{s_1}, \frac{s_3}{s_2}\right)\right]\setminus\{1\}\neq \emptyset.
\end{equation*}
\end{Lemma}

\begin{remark} 
In \cite{ANP} a slightly weaker version of Lemma \ref{LemmaF} was established and the notation was slightly different. Specifically, in \cite[Lemma 3.4]{ANP} it is assumed that 
the eigenvalues of $F$ and $G$ satisfy, in addition, 
\begin{equation}\label{wo}
 f_1\leq g_1\leq  g_2\leq g_3\leq f_3.
\end{equation}
The proof of Lemma \ref{LemmaF} follows the lines of the proof of \cite[Lemma 3.4]{ANP} with the only difference given by the formulas of $\al_+(F,G)$ and  
$\beta_-(F,G)$, which, in the case when \eqref{wo} is satisfied, read as 
\begin{align*} 
& \al_+ (F,G)= \min\left(\frac{g_2}{f_2},\frac{g_3}{f_3}\right) ,
&\beta_-(F,G)= \max \left(\frac{g_1}{f_1},\frac{g_2}{f_2}\right).
\end{align*}
\end{remark}
\begin{remark}\label{not-uni}
One may wonder whether Lemma \ref{LemmaF} misses some interesting rank-one connections by assuming $F\in \mathcal D_d$. This is not the case. Indeed, one can check that if both $F$ and $G$ are uniaxial, so that Lemma \ref{LemmaF} cannot be applied, then only one possible rank-one connection exists and it requires $n$ be a vector of the canonical basis  and that the double eigenvalue be the smallest for both matrices, or the largest for both matrices. These are well-known connections, already studied in \cite{ACLM}, and can be visualised with the help of Figure \ref{fig:sextant-annotated} as segments  lying either on the vertical dotted line or on the oblique dotted line. 
For this reason, we may disregard the case when $F\notin \mathcal D_d$.
\end{remark}
\section{The set $\L$}\label{sec-t2-bis}
We recall the definition of the set $\L$, which provides an inner bound for the set $K^{*}(S)$ 
(see \cite[\S 5]{ANP} for more details).
We start by introducing the uniaxial points $U_{\alpha}$ and 
$U_{\beta}$, and the rotations $R_\alpha$ and $R_\beta$ (see Definition \ref{def_U}).
We then define $\Gamma_{\alpha}$ and $\Gamma_{\beta}$ as the projection on $\H$ of the rank-one segments connecting a specific multiple of $S$ to 
$R_\alpha^t U_{\alpha}R_\alpha$ and $R_\beta^t U_{\beta}R_\beta$ respectively (Definition \ref{def33}).  
The set $\L$ is finally defined as the subset of $\H$ enclosed by  $\Gamma_{\alpha}$, $\Gamma_{\beta}$ and appropriately reflected and rotated copies of $\Gamma_{\alpha}$ and $\Gamma_{\beta}$ (see Definitions \ref{def:Gamma-construction}, \ref{matL} and  Figures~\ref{fig:sextant-annotated}-\ref{fig:hexagon-annotated}).
\begin{definition}\label{def_U}
    Set \begin{equation}\label{upm}
\begin{array}{cc}
U_{\alpha}=
\left(
\begin{array}{ccc}
u_{\alpha}&0&0\\
0&u_{\alpha}&0\\
0&0&1-2 u_{\alpha}
\end{array}
\right), 
\quad U_{\beta}=
\left(
\begin{array}{ccc}
1-2 u_{\beta}&0&0\\
0&u_{\beta}&0\\
0&0&u_{\beta}
\end{array}
\right),
\end{array}
\end{equation}
where $u_{\alpha}$ and $u_{\beta}$ are the smallest and largest root of 
\begin{equation}\label{rootU}
b(x):=6 s_2\,x^2+ x\,(s_1 s_3-3 s_2-4 s_2^2)+ 2 s_2^2
\end{equation}
respectively.
Set
\begin{equation}\label{mpd1}
\begin{array}{lll}
n_\alpha=\left(
\begin{array}{c}
\cos \varphi_\alpha\\0\\ \sin \varphi_\alpha\end{array}
\right),\quad
R_\alpha=
\left(
\begin{array}{ccc}
0&-1&0\\
-\cos \theta_\alpha &0&\sin \theta_\alpha\\
-\sin \theta_\alpha &0&-\cos  \theta_\alpha
\end{array}
\right),\\\\
n_\beta=\left(
\begin{array}{c}
\cos \varphi_\beta\\0\\ \sin \varphi_\beta\end{array}
\right),\quad
R_\beta=
\left(
\begin{array}{ccc}
\cos \theta_\beta&0&-\sin \theta_\beta\\
\sin \theta_\beta&0&\cos  \theta_\beta\\
0&-1&0
\end{array}
\right),
\end{array}
\end{equation}
with 
\begin{equation}\label{mpd2}
    \cos (2\varphi_\alpha)=\frac{s_2(s_1+s_3)+u_\alpha(2 s_1 s_3-s_1-s_3)}{(s_3-s_1)(s_2-u_\alpha)},
     \, \cos (2\varphi_\beta)=\frac{s_2(s_1+s_3)+u_\beta(2 s_1 s_3-s_1-s_3)}{(s_3-s_1)(s_2-u_\beta)},
\end{equation}
\begin{equation}    \label{mpd3}
\begin{array}{ll}
\displaystyle{\cos(2 \theta_\alpha) =\frac{u_\alpha(s_3-s_1)+(u_\alpha
-s_2) \cos(2 \varphi_\alpha)}{s_2(1-3 u_\alpha)},\cos(2 \theta_\beta) =-\frac{u_\beta(s_3-s_1)+(u_\beta
-s_2) \cos(2 \varphi_\beta)}{s_2(1-3 u_\beta)}.}
\end{array}
\end{equation}
\end{definition}
The next proposition collects results from \cite{ANP} (see \cite[Proposition 5.3]{ANP}), characterising the matrices of $U_\alpha$ and $U_\beta$ 
as solutions of a specific matrix equation.
Set 
\begin{equation}\label{alfabeta}
\alpha:=  \frac{u_{\alpha}}{s_2},\quad  \beta:=  \frac{u_{\beta}}{s_2}.
\end{equation}
\begin{proposition}\label{prop_opt_curve}
Let $U_\alpha, U_\beta, R_\alpha, R_\beta, n_\alpha, n_\beta$ be defined by \eqref{upm}-\eqref{mpd1}.
Then 
\vspace{2mm}
\noindent
\begin{itemize}
\item[(i)]
 $u_\alpha$ and $u_\beta$ satisfy
$ s_1\leq   u_{\alpha}<\frac 1 3< u_{\beta}\leq  s_3$;
\vspace{2mm}
\item[(ii)]
$
\alpha_-(S,U_{\alpha})=\alpha_+(S,U_{\alpha}) = \alpha,\quad
\beta_-(S,U_{\beta})= \beta_+(S,U_{\beta}) = \beta;
$
\vspace{2mm}
\item[(iii)]
the matrices $U_\alpha, U_\beta$  are the unique solutions to
\begin{equation}\label{2791}
 R_\alpha^t  U_{\alpha}R_\alpha = \alpha S +(1-\alpha)\, n_\alpha\otimes n_\alpha,
\end{equation}
\begin{equation}\label{2792}
 R_\beta^t  U_{\beta}R_\beta = \beta S +(1-\beta)\, n_\beta\otimes n_\beta,
\end{equation}
respectively.
 \end{itemize}
\end{proposition}
The matrices $U_\alpha$ and $U_\beta$  correspond to the blue and orange points in  Figure~\ref{fig:sextant-annotated}.
In order to define $\Gamma_\alpha$ and $\Gamma_\beta$ we will need an efficient way to describe curves  in eigenvalue space. We therefore introduce some notation that 
will be used in Section \ref{More} as well.
Let $F, G\in K^*(S)$ and assume that 
\eqref{oldbp} holds.
For each $\lambda\in A(F,G)$ and corresponding $n$, see \eqref{ni}, 
consider the one-parameter family of unit-trace matrices
\begin{equation}\label{s11}
\begin{array}{l}\displaystyle{M_{F,G}(\lambda,t) 
=\eta(\lambda,t) F+(1-\eta(\lambda,t)) n\otimes n,\quad   \eta(\lambda,t):= \frac{\lambda}{\lambda + t(1-\lambda)}},
\end{array}
\end{equation}
and note that
\begin{equation*}
    M_{F,G}(\lambda,0)=F,\quad M_{F,G}(\lambda,1)=R^tGR.  
\end{equation*}
We now specialize \eqref{s11} to the case when  $F=S,$ and $\lambda=\alpha$ or 
$\lambda=\beta$, see \eqref{alfabeta},  and set
\begin{equation}\label{Gamma}
\begin{array}{l}
M_\alpha(t):=M_{S,U_\alpha}(\alpha,t),\quad
M_\beta(t):=M_{S,U_\beta}(\beta,t). 
\end{array}
\end{equation}
We recall that 
\begin{equation}\label{endpoints}
M_\alpha(1)=R_\alpha^t U_\alpha R_\alpha,\quad M_\beta(1)=R_\beta^t U_\beta R_\beta.
\end{equation} 
\begin{definition}\label{def33}
The curves $\Gamma_\alpha, \Gamma_\beta$ are defined as follows. Consider the parametric curves
\begin{equation}\label{m1m2m3}
   t\to m_\alpha(t):=(m_1(\alpha,t),m_2(\alpha,t),m_3(\alpha,t)),   \quad t\in[0,1],  
\end{equation}
\begin{equation}\label{m1m2m3p}
       t\to m_\beta(t):=(m_1(\beta,t),m_2(\beta,t),m_3(\beta,t)), \quad t\in[0,1],
\end{equation}
where $ m_i(\alpha,t), \, m_i(\beta,t) $ denote
the ordered  eigenvalues of $M_\alpha(t)$ and $M_\beta(t)$.
Then
   \begin{equation*}
\Gamma_{\alpha}:=m_\alpha([0,1]),\quad 
\Gamma_{\beta}:=m_\beta([0,1]).
  \end{equation*}
\end{definition}

The curves $\Gamma_\alpha, \Gamma_\beta$ are shown in Figure~\ref{fig:sextant-annotated}.

\begin{figure}\label{fig:curves-annotated}
\begin{subfigure}{0.33\textwidth}
\includegraphics[width=\textwidth]{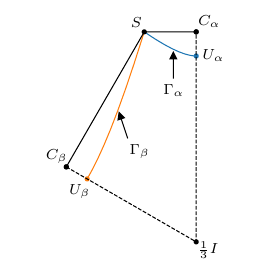}\caption{}\label{fig:sextant-annotated}
\end{subfigure}%
\begin{subfigure}{0.33\textwidth}
\includegraphics[width=\textwidth]{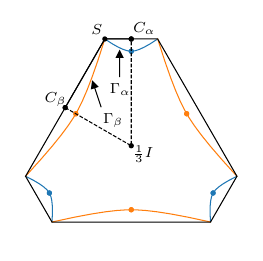}\caption{}\label{fig:hexagon-annotated}
\end{subfigure}%
\begin{subfigure}{0.33\textwidth}
\includegraphics[width=\textwidth]{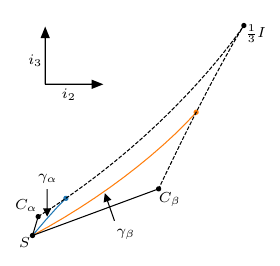}\caption{}\label{fig:invariants-annotated}
\end{subfigure}
\caption{
The left figure shows important fields in one sextant ($s_1\le s_2 \le s_3$) of the unit-trace plane. The outer quadrilateral connects the field $S$ to the isotropic field $\frac{1}{3}I$ and the two uni-axial points, $C_{\alpha}=(\frac{s_1+s_2}{2},\frac{s_1+s_2}{2},s_3)$ and $C_{\beta}=(s_1,\frac{s_2+s_3}{2},\frac{s_2+s_3}{2})$. The curves $\Gamma_{\alpha}$ and $\Gamma_{\beta}$ from Definition~\ref{def33} are also shown, together with their intersections with the uniaxial lines (the dashed lines). The center figure shows the construction in Definition~\ref{def:Gamma-construction}. The set $\L$ is enclosed by the union of the reflected copies of $\Gamma_\alpha, \Gamma_\beta$. The rightmost figure shows the fields in the sextant once again, but in the plane of the matrix invariants, $(i_2,i_3)$, which are used in the proof of the stability.
Abusing notation we denote points by the same letters in (a)-(c).
}
\end{figure}
\begin{definition}\label{def:Gamma-construction}
The closed curve $\Gamma$ is obtained as follows.
First, reflect $\Gamma_{\alpha}$ along the line $m_1=m_2$ in the plane $m_1+m_2+m_3=1$, then consider the union of the curves obtained with its $2\pi/3$ rotations within the unit trace plane.
Next, reflect $\Gamma_{\beta}$ along the line $m_2=m_3$ in the plane $m_1+m_2+m_3=1$, and consider the union of the curves obtained with its $2\pi/3$ rotations within the unit trace plane.
Finally, $\Gamma$ is the union of the six curves thus defined.  
\end{definition}
\noindent The construction of $\Gamma$ is shown in Figure~\ref{fig:hexagon-annotated}. 
By construction, $\Gamma$ is a simple closed Jordan curve, and, for future reference, we note that  %
\begin{equation}\label{m_2}
    m_2(\alpha,t) =s_2 \eta(\alpha,t),\quad m_2(\beta,t) =s_2 \eta(\beta,t),
\end{equation}
 which readily follows from the fact that the second component of $n_\alpha$ and $n_\beta$ are zero.
\begin{definition}\label{matL}
    We denote by $\L$ the bounded closed subset of $\H$ enclosed by $\Gamma$ (see  Figure~\ref{fig:hexagon-annotated}).
\end{definition}
\begin{theorem}[\cite{ANP}, Theorem 5.7]
 The set $\L$ of Definition \ref{matL} satisfies 
 $\L\subseteq K^*(S)$.
\end{theorem}
\section{Stability under lamination of the set $\L$}\label{More}
We first recall the definition of stability under lamination for a subset of $K^*(S)$. Then, in Proposition \ref{prop_opt_curve_2ndpart}, we provide an equivalent description in terms of geometrical quantities. 
\par
\begin{definition}\label{def-stability}
We say that a subset $L$ of $K^*(S)$ is stable under lamination if the following holds. 
For any pair $F,G\in  L$ such that 
$\rank(\lambda F- R^tGR$) =1 for some $\lambda\in\R$ and $R\in SO(3)$, the projection on $K^*(S)$ of the rank-one segment connecting $\lambda F$ and $R^tGR$ belongs to $L$. Specifically, if $t\mapsto M_{F,G}(\lambda,t)$ is defined by \eqref{s11}, then the curve 
$t\mapsto m_{F,G}(\lambda,t)$, with $m_{F,G}$ the triple of the corresponding ordered eigenvalues, satisfies  
$m_{F,G}(\lambda,t) \in L$ for each $t\in[0,1]$.
\end{definition}
\begin{theorem}\label{th.stability_1}
The set $\L$ is stable under lamination.
\end{theorem}
\begin{proof}
This follows from Propositions \ref{prop_opt_curve_2ndpart}, \ref{propa}, and \ref{propb}.
\end{proof}

\begin{remark}
The set $K^*(S)$ is stable under lamination (see, e.g., \cite{Mnotes}, \cite{Zhang}).  
As already observed in the introduction, we do not know whether $\L=K^*(S)$ and therefore the proof of the stability of $\L$ will require a careful examination of the rank-one connections between its points. 
\end{remark}
We find it convenient to change coordinates. Rather than using the eigenvalues $m_1,m_2,m_3$, we use the symmetric invariants.
The change of variable \begin{equation*}(m_1,m_2,m_3)\mapsto(i_1,i_2,i_3):=(m_1+m_2+m_3,\, m_1 m_2+m_2 m_3+ m_3 m_1,\,m_1 m_2 m_3)\end{equation*} is a diffeomorphism, except where two eigenvalues collapse. Recall that we work with symmetric matrices $M$ with unit trace, so that, for all of them  $i_1(M)=1$.
Let $F\in \mathcal D_d$ and $G\in \mathcal D$, then, using the notation in  \eqref{s11} we set 
\begin{align}\label{dernot}
&  x(\lambda,t,F,G):= i_2(M_{F,G}(\lambda,t)),\quad   y(\lambda,t, F, G):= i_3(M_{F,G}(\lambda,t)),\\
\label{dernot2} 
&  \dot x(\lambda,t,F,G)=\frac{\partial x}{\partial t}(\lambda,t, F, G),\quad 
    \dot y(\lambda,t, F, G)=\frac{\partial y}{\partial t}(\lambda,t, F, G),
\end{align}
and,
\begin{equation}\label{xF}
    x_F:=x(\lambda,0,F,G),\quad x_G:=x(\lambda,1,F,G),\quad
   y_F:=y(\lambda,0,F,G),\quad y_G:=y(\lambda,1,F,G) . 
\end{equation}
See Figure~\ref{fig:invariants-annotated} for plot of important fields in these coordinates. The curves $\Gamma_{\alpha}, \Gamma_{\beta}$ will be denoted by $\gamma_{\alpha}, \gamma_{\beta}$ in the $(x,y)$-space, as clarified by the following definition.
\begin{definition}\label{defgamma}
For $M_\alpha$ and $M_\beta$ given by \eqref{Gamma} and $t\in[0,1]$, set  
$$\Phi_\alpha(t):=(i_2(M_\alpha(t)),i_3(M_\alpha(t))),\,
\Phi_\beta(t):=(i_2(M_\beta(t)),i_3(M_\beta(t))).$$
Then 
\begin{equation}\label{gamma-+}
  \gamma_\alpha:=\Phi_\alpha([0,1]),\quad \gamma_\beta:=\Phi_\beta([0,1]).
 \end{equation}
\end{definition}
Abusing notation, we identify points in matrix space with points in invariant space, so we can write $ M_\alpha(t)\in \gamma_{\alpha}$. 
Similarly, the image of $\L$ via the map $(i_2,i_3)$ is still denoted by $\L$.
\par
The following lemma ensures that  $\gamma_\alpha$ and $\gamma_\beta$ on the $(i_2,i_3)$-plane are exactly as depicted in Fig. \ref{fig:invariants-annotated}. Its proof is postponed to Section \ref{posticipato}.
\begin{Lemma}\label{davveroultimo}
The sets  $\gamma_\alpha$ and $\gamma_\beta$ are graphs of functions and intersect only at the point $(i_2(S),i_3(S))$. At such point the unit tangent vector to $\gamma_\alpha$ is obtained by a counterclockwise rotation by an angle of amplitude less than $\pi/2$ of the unit tangent vector to $\gamma_\beta$. 
\end{Lemma}
\begin{proposition}\label{prop_opt_curve_2ndpart}
The set $\L$ given by Definition \ref{matL}
is stable under lamination if and only if the following geometrical property holds.   Let
$F, G$ both belong to either $\gamma_\alpha$ or $\gamma_\beta$.
Then the tangent vector at $F$ to any rank-one trajectory that starts at $F$  and ends at  $G$, points strictly inwards $\L$, except for the optimal rank-one trajectory that lies exactly on $\gamma_\alpha$ or $\gamma_\beta$.
\end{proposition}
\begin{proof}
The condition is clearly necessary. 
We focus on the sufficiency. 
 Let $F,G\in\L$ be such that  $\rank(\lambda F- R^tGR) =1$ for some $\lambda\in\R$ and $R\in SO(3)$. 
 From now on we will drop the subscript $F,G$ from all the variables.
 Let $m(\lambda,t)$ be as in Definition \ref{def-stability} and 
assume by contradiction that the set $m(\lambda,[0,1])$ intersects $\L^c$, the complement of $\L$. 
Since $m(\lambda,t)$ may cross 
$\partial \L$ only a finite number of times, say for $t=t_i, \, i=1,\dots N$, one can decompose the trajectory into a finite number of arcs $m(\lambda,[t_{i-1},t_i])$, each lying either inside $\L$, or outside $\L$ except for its endpoints which lie exactly on $\partial\L$. In the first case there is nothing to prove. 
In the second case we observe that $m(\lambda,t_{i-1})$ and  $m(\lambda,t_i)$ both belong to the same branch of 
$\partial \L$, namely, 
$m(\lambda,t_{i-1}),m(\lambda,t_{i})\in\gamma_{\alpha}$ or 
$m(\lambda,t_{i-1}),m(\lambda,t_{i})\in\gamma_{\beta}$. This is 
a consequence of the bound
$s_1\leq m_j\leq s_3$ for each $j=1,2,3$ and $t\in[0,1]$. Then a necessary and sufficient condition for the arc starting at $m(\lambda,t_{i-1})$ and ending at $m(\lambda,t_i)$ to lie in $\L$ is that the tangent vector
at  $m(\lambda,t_{i-1})$  points inwards. 
\end{proof}
\begin{remark}\label{remutile}
Using Lemma \ref{davveroultimo}, one can see that 
the geometric condition expressed in Proposition \ref{prop_opt_curve_2ndpart}, namely the 
condition for the tangent vector to point inwards, can be equivalently expressed by saying that if $F$ and $G$ belong to $\gamma_\alpha$, with $F$ closer to $S$, then the maximum of the slope of the tangent vector at $F$ among
 all the  rank-one trajectories that start at $F$ and end at $G$ is attained when the trajectory is exactly the arc of $\gamma_\alpha$ joining $F$ and $G$ 
 (see Proposition \ref{propa}). When $F$ and $G$ belong to $\gamma_\beta$, with $F$ closer to $S$, 
 then the trajectory that lies on $\gamma_\beta$ minimises the slope among all possible rank-one trajectories connecting $F$ and $G$ (see Proposition \ref{propb}).
\end{remark}
We conclude the present section by outlining the main steps towards the proof of the geometric conditions described in Remark \ref{remutile}, namely of
Propositions \ref{propa} and \ref{propb}, on which 
Theorem \ref{th.stability_1} is based.
\par
\noindent
{\em Step 1.}
Let $F\in \mathcal D_{\rm d}$ and let $G\in {\mathcal T^{1}(F)}$, see Definition  \ref{MT}, so that the set  $A(F,G)\notequal \emptyset$. 
For any $\lambda \in A(F,G)$, consider the smooth trajectory originating at $F$ and ending at $G$, see \eqref{s11}. 
Assuming $\dot{x}(\lambda,0,F,G)\neq 0$, define the ``normalised slope'' in the invariants space as
\begin{equation}\label{slopeE}
\tau(\lambda,F,G):=\left(\frac{\dot{y}(\lambda,0,F,G)}{\dot{x}(\lambda,0,F,G)}\right)\left(\frac{x_G}{y_G}\right),
\end{equation}
with $x_G$ and $y_G$ given by \eqref{dernot}-\eqref{xF}.
 
\begin{remark}
When $\dot{x}(\lambda,0,F,G)\neq 0$, the actual slope of the tangent vector at the initial point $F$ to the rank-one trajectory that starts at $F$ and ends at $G$, namely $t\to (x(\lambda,t,F,G),y(\lambda,t,F,G))$,  
is given by $\frac{\dot{y}(\lambda,0,F,G)}{\dot{x}(\lambda,0,F,G)}$.  With a slight abuse of notation, we use this terminology for $\tau(\lambda,F,G)$ instead. 
    \end{remark}
For $\dot{x}(\lambda,0,F,G)\neq 0$, we compute \eqref{slopeE}  obtaining the following
\begin{equation}\label{tau0}
\tau(\lambda,F,G)=
\displaystyle{
\frac{1-\lambda^2(3-2\lambda) \frac{y_F}{y_G}}{\lambda\bigl(1-\lambda(2-\lambda) \frac{x_F}{x_G}\bigr)}  
},
\end{equation}
see Proposition \ref{xy}.

\par
\noindent
{\em Step 2.} Motivated by  Proposition \ref{prop_opt_curve_2ndpart}, 
we specialize \eqref{tau0} to pairs $F,G$ that lie either on $\gamma_{\alpha}$ or $\gamma_{\beta}$.
We will proceed by considering in detail the case when $F,G$ lie on $\gamma_\alpha$. We will hint at the necessary modifications 
for the case of $\gamma_\beta$ in Section \ref{final7}.
Thus, we choose $F=M_{\alpha}(p)$, $G=M_{\alpha}(q)$, with  $M_\alpha$ given by \eqref{Gamma} and $0\leq p\leq q\leq 1, p\neq 1.$
We prove   that, 
\begin{equation}\label{taubis}
1-\lambda(2-\lambda) \frac{x_{M_{\alpha}(p)}}{x_{M_{\alpha}(q)}}\neq 0, \quad \forall \, 
\lambda\in A(M_{\alpha}(p),M_{\alpha}(q))
\end{equation}
and 
\begin{equation}\label{bohhh}
     \tau(\lambda,M_{\alpha}(p),M_{\alpha}(q))=h(\lambda,p,q),
     \end{equation} 
where $h$ is defined in \eqref{hhh} and depends only on the second eigenvalue $m_2(\alpha,t)$ computed for $t=p$ and $t=q$, see \eqref{m_2}.
\par
\noindent
{\em Step 3.} We prove that the slope is maximum when 
$\lambda = \alpha_+\big(M_{\alpha}(p),M_{\alpha}(q)\big)$, namely, that for each $0\leq p\leq q\leq 1,\,p\neq 1$ one has 
\begin{equation}\label{slopeEoptimal}
  \max_{\lambda\in A(M_{\alpha}(p),M_{\alpha}(q))} 
h\left(\lambda,p,q\right)=
h\Big(\alpha_+\big(M_{\alpha}(p),M_{\alpha}(q)\big),p,q\Big).
\end{equation}
\noindent
This follows from Proposition \ref{propa} and Lemma \ref{83}, which imply that the slope on the right hand side of \eqref{slopeEoptimal} corresponds to the rank-one trajectory joining $M_{\alpha}(p)$ and $M_{\alpha}(q)$ that lies on $\gamma_{\alpha}$.
Notice that $x_{M_{\alpha}(q)}/y_{M_{\alpha}(q)}$ is positive, therefore \eqref{slopeEoptimal} is equivalent to maximizing the actual slope $\dot y/\dot x$. 

\section{Proving Steps 1-3}
\subsection{Step 1: proving \eqref{tau0}}

The next result gives a compact form for $\tau(\lambda,F,G)$.
Recall Definition  \ref{MT}.
\begin{proposition}\label{xy}
Let $F\in \mathcal D_{\rm d}$ and $G\in {\mathcal T^{1}(F)}$. 
Then \eqref{tau0} holds for each  $\lambda\in A(F,G)$.
\end{proposition}
\begin{proof}
We use \eqref{s92} of  Lemma \ref{usid} and for each $t\in [0,1]$ we compute 
 \begin{equation}\label{xdotydot}
        \begin{array}{ll}
\dot{x}(\lambda,t,F,G) &=     
\frac{\partial \eta}{\partial t}(\lambda,t)\,
         \left(\frac{x_G-\la^2x_F}{\la(1-\la)}-\frac{2(x_G-\la x_F)}{\la(1-\la)} \eta(\lambda,t)\right) 
        \\\\
        &=\frac{1}{\la^2}\bigl(2(x_G-\la x_F) \,\eta^3(\la,t)-
        (x_G-\la^2 x_F)\, \eta^2(\la,t)\bigr)
         \\\\
    \dot{y}(\lambda,t,F,G) & =
    \,\frac{\partial \eta}{\partial t}(\lambda,t)\,\left( \frac{2(y_G-\la^3y_F)}{\la^2(1-\la)}\,\eta(\lambda,t) -  \frac{3(y_G-\la^2y_F)}{\la^2(1-\la)}\eta^2(\lambda,t)\right)
    \\\\
&=\frac{1}{\la^3}\bigl(3(y_G-\la^2 y_F) \,\eta^4(\la,t)-
        2(y_G-\la^3 y_F)\, \eta^3(\la,t)\bigr).\\\\
        \end{array}  
    \end{equation}  
Thus
\begin{equation*}%
         \dot{x}(\lambda,0,F,G)=
         \frac{1}{\lambda^2}\bigl(x_G - \lambda(2-\lambda) x_F\bigr)
         ,\quad
          \dot y(\lambda,0,F,G) = \frac{1}{\lambda^3}\bigl(y_G - \lambda^2(3-2\lambda)y_F\bigr).
    \end{equation*}
In particular \eqref{tau0} holds when $\dot{x}(\lambda,0,F,G)\neq0$.
\end{proof}

\subsection{Step 2: proving \eqref{bohhh}}

Consider the family $t\mapsto M_\alpha(t)$ and
denote by
$\mathcal E_j(t)$   the  ordered eigenvalues of $M_\alpha(t)$ (see  \eqref{m1m2m3}, \eqref{m_2}), namely 
\begin{equation}\label{eigen}
\mathcal E_j(t):=m_j(\alpha,t).
\end{equation}
Define
\begin{equation}\label{6.15}
\begin{array}{l}
r(p,q):=\displaystyle{\frac{\E_2(p)}{\E_2(q)}\left(\frac{2u_\alpha - \E_2(p)}{2u_\alpha - \E_2(q)}
\right)}, \quad
            s(p,q):=\displaystyle{\frac{\E_2^2(p)}{\E_2^2(q)}\left(
            \frac{3u_\alpha - 2\E_2(p)}{3u_\alpha - 2 \E_2(q)}\right)}
            \end{array}
\end{equation}
\begin{equation}\label{hhh}
h(\lambda,p,q):=
    \frac{1-\lambda^2(3-2\lambda)\,s(p,q)}{\lambda\bigl(1-\lambda(2-\lambda)\,r(p,q)\bigr)}.
\end{equation}
The next result shows that  $r$, $s$ and $h$ are well-defined.
\begin{proposition}\label{lemmainvarianti}
 For each $0\leq p \leq q \leq 1,\,p\neq 1$ and $\lambda\in A(M_{\alpha}(p),M_{\alpha}(q))$, we have 
\begin{equation}\label{wd}
3u_\alpha -2 \E_2(q) >0, \quad
1-\lambda(2-\lambda)r(p,q) >0.
\end{equation}
Moreover, \eqref{taubis} and \eqref{bohhh} hold. 
\end{proposition}
\begin{proof}
We show the first inequality in \eqref{wd}. Since $q\mapsto \E_2(q)\leq s_2$, it is enough to check that  $u_\alpha>\frac 2 3 s_2$. To this end, we use the function $b$ defined in \eqref{rootU}. We observe that
$b(\frac 2 3 s_2)=\frac{2 s_1 s_2 s_3}{3}>0$ and
$b'(\frac 2 3 s_2)=-2s_1 s_2 -3s_2^2 -3s_2 s_3<0$, thus $\frac 2 3 s_2<u_\alpha$, since $u_\alpha$ is the smallest root of the quadratic polinomial $b(x)$. 
\par
We proceed to check the second inequality in \eqref{wd}. We distinguish two cases. If $p=q$, then the statement is obvious. So we assume $p<q\leq1$. Since $\lambda\mapsto 1-\lambda(2-\lambda)\,r(p,q)$ is a quadratic polynomial in the variable  $\lambda$, it is minimised at $\lambda=1$. Hence, it suffices to show that   $1-\,r(p,q) >0$. 
By definition, this is the same as showing that
$$
(\E_2(p)-\E_2(q))(\E_2(p)+\E_2(q)-2u_\alpha)>0.$$ 
 Using the strict monotonicity of $\E_2$, we obtain 
$$
\E_2(p)-\E_2(q)> 0,\quad \E_2(p)+\E_2(q)-2u_\alpha  > 2\E_2(1)-2u_\alpha=0.
$$
Formula \eqref{taubis} follows from \eqref{wd} and \eqref{mainxyz} of Lemma \ref{usid}.
Finally, \eqref{bohhh} follows from \eqref{mainxy} of Lemma \ref{usid}. 
\end{proof}
 Set
 \begin{equation}\label{Lambda}
      \Lambda(p,q):=\frac{\E_2(q)}{\E_2(p)},
 \end{equation}
 and remark that $ \Lambda(p,q)\leq 1$ and  $\Lambda(p,q)= 1$ if and only if $p=q$.
The next lemma implies that $ h(\Lambda(p,q),p,q)$ represents  
the slope of the tangent vector at the point $M_{\alpha}(p)$ when running a trajectory from $M_{\alpha}(p)$ to $M_{\alpha}(q)$ on the restriction of $\gamma_{\alpha}$ to the arc joining $M_{\alpha}(p)$ and $M_{\alpha}(q)$. 
In addition, it gives an explicit formula for the extremes of 
the interval  $A_\alpha (M_{\alpha}(p),M_{\alpha}(q))$ defined in \eqref{A(F,G)}.  
\begin{Lemma}\label{83}
Let  $0\leq p \leq q \leq 1,\,p\neq 1$ and let $\Lambda(p,q)$ be defined in \eqref{Lambda}. There exist $R_p,R_q\in SO(3)$ such that 
\begin{equation}\label{crucial}
    \begin{array}{cc}
       R_q^t M_{\alpha}(q) R_q = \Lambda(p,q) R_p^t M_{\alpha}(p) R_p+(1-\Lambda(p,q))n_\alpha\otimes n_\alpha.
    \end{array}
\end{equation}   
\noindent 
In particular $\Lambda(p,q)\in A(M_{\alpha}(p),M_{\alpha}(q))$.
Moreover
\begin{equation}\label{ameno}
\alpha_-(M_{\alpha}(p),M_{\alpha}(q))=\max\left\{\frac{ {\mathcal E_1(q)}}{\E_2(p)},\frac{\E_2(q)}{\mathcal E_3(p)}\right\},
\quad
\alpha_+(M_{\alpha}(p),M_{\alpha}(q)) = \Lambda(p,q),
\end{equation}
and
\begin{equation}\label{974}
 q<1 \Longrightarrow  \alpha_-(M_{\alpha}(p),M_{\alpha}(q))<\alpha_+(M_{\alpha}(p),M_{\alpha}(q)).
\end{equation}
\end{Lemma}
\begin{proof} 
By definition of $\Gamma_{\alpha}$, see \eqref{Gamma}, there exist two rotations  $R_p, R_q$   such that  
\begin{equation}\label{pq}
    \left\{\begin{array}{cc}
       R_p^t M_{\alpha}(p) R_p = \eta(p) S+(1-\eta(p))n_\alpha\otimes n_\alpha  \\\\
       R_q^t M_{\alpha}(q) R_q = \eta(q) S+(1-\eta(q))n_\alpha\otimes n_\alpha    ,
    \end{array}
    \right.
\end{equation}
where we set
 \begin{equation}\label{not1}
\eta(t):= \eta(\alpha,t),
    \end{equation}
    with $\eta(\alpha,t)$ defined in \eqref{s11}.
Then, by eliminating $S$ from \eqref{pq} and using \eqref{Lambda}, we deduce \eqref{crucial}.
The first equality in \eqref{ameno} holds by definition, since $p\neq 1$ implies that $M_\alpha(p)\in \mathcal D_d$, and hence Lemma \ref{LemmaF} can be applied.
We now prove the second equality in \eqref{ameno}. 
 Recall that by construction $\alpha_+\leq 1$. If $p=q$, the statement is immediate, since
$\alpha_+(M_{\alpha}(p),M_{\alpha}(p)):=\min_{j}\left\{\frac{\E_j(p)}{\E_j(p)}\right\}=1=\Lambda(p,p)$.
Now let $p<q$ and recall that in this case $\alpha_+(M_\alpha(p),M_\alpha(q))<1$. By the first part of the present lemma  $\Lambda(p,q)\in A_\alpha (M_\alpha(q),M_\alpha(p))$ and thus
$A_\alpha(M_\alpha(p),M_\alpha(q))\neq \emptyset$.
We have 
\begin{equation*}
     \alpha_+(M_\alpha(p),M_\alpha(q)):=\min_{j}\left\{\frac{\E_j(q)}{\E_j(p)}\right\}\leq \frac{\E_2(q)}{\E_2(p)}%
     =\Lambda(p,q)\leq\alpha_+(M_\alpha(p),M_\alpha(q)) \,.
\end{equation*}
To prove \eqref{974} recall
that in general $\alpha_-(M_\alpha(p),M_\alpha(q))\leq \alpha_+(M_\alpha(p),M_\alpha(q))$.
Equality occurs if and only if for some $p$ and $q$ one has
\begin{equation*} \hbox{either} \quad 
    \frac{\E_1(q)}{\E_2(p)}=\frac{\E_2(q)}{\E_2(p)}\quad
     \hbox{or } \quad
    \frac{\E_2(q)}{\E_3(p)}=\frac{\E_2(q)}{\E_2(p)}\,.
\end{equation*}
Hence, if equality occurs, $M_\alpha(p)$ or $M_\alpha(q)$ must be uniaxial. Since $M_\alpha(p)$ and $M_\alpha(q)$ belong to the optimal trajectory $\gamma_{\alpha}$, the only uniaxial point on $\gamma_{\alpha}$ is $U_{\alpha}$ which corresponds to  $q=1$.
\end{proof}
\subsection{Step 3: proving \eqref{slopeEoptimal}} 
 We state the main result of the present section.
\begin{proposition}\label{propa}
   If $M_{\alpha}(p),M_{\alpha}(q)\in\gamma_{\alpha}$, then for each $0\leq p \leq q \leq 1$, with $p\neq 1$,
 \begin{equation}\label{conj1a}
 \max_{\lambda \in A(M_{\alpha}(p),M_{\alpha}(q))} h(\lambda,p,q)= h(\Lambda(p,q),p,q).
 \end{equation}   
\end{proposition}
\begin{proof}
To ease notation we set, for $\lambda>0$,
 \begin{equation}
\begin{array}{cc}
     c(\lambda):= \lambda^2(3-2\lambda)  ,&
    d(\lambda):= \lambda(2-\lambda),
 \end{array}
     \end{equation}
and write \eqref{hhh} as
$$
h(\lambda,p,q)=\left(\frac 1 \lambda\right)
\frac{1- c(\lambda)\,s(p,q) }{1-d(\lambda)\,r(p,q)}.
$$ 
Formula \eqref{conj1a} is equivalent to
\begin{equation}\label{ine1}
h(\lambda,p,q)\leq \left(\frac{1} {\Lambda(p,q)}\right)
\frac{1- c(\Lambda(p,q))\,s(p,q) }{1-d(\Lambda(p,q))\,r(p,q)},
    \end{equation}
for each $0\leq p \leq q \leq 1$, with $p\neq 1$, and for each $\lambda\in A(M_{\alpha}(p), M_{\alpha}(q))$.    
We use formula \eqref{mainxyz} in Lemma \ref{usid} and the definition \eqref{Lambda} of $\Lambda(p,q)$ to write the right-hand side  of \eqref{ine1} as follows: 
    \begin{equation}\label{rhsine2}
    \begin{array}{cc}
        \displaystyle{\left(\frac{1}{\Lambda(p,q)}\right)
\frac{1- c(\Lambda(p,q))\,s(p,q) }{1-d(\Lambda(p,q))\,r(p,q)}  =
3\left(\frac{\E_2(p)}{\E_2(q)}\right)\frac{2u_\alpha-\E_2(q)}{3u_\alpha-2\E_2(q)}}.
    \end{array}
    \end{equation}
Motivated by the previous calculation, we compute the left-hand side of \eqref{ine1}, using the variable $X$, with \begin{equation}\label{X}
        X:= \lambda \frac{\E_2(p)}{\E_2(q)}=\frac{\lambda}{\Lambda(p,q)}.
    \end{equation}
    A long but straightforward algebraic manipulation yields 
\begin{equation} \label{hoggi}
h\left(\lambda,p,q\right)=h\left(\frac{\E_2(q)}{\E_2(p)}\,X,p,q\right)=
   \left(\frac{\E_2(p)}{\E_2(q)}\right)\left(\frac{ 2u_\alpha -\E_2(q)}{3u_\alpha -2\E_2(q)}\right) \frac{P(X)}{Q(X)},  
\end{equation}
with 
\begin{equation*}
\left\{
\begin{array}{ll}
   Q(X)=X\Bigl(\E_2(p)(2u_\alpha -\E_2(q)) - 2 X \E_2(p)(2u_\alpha - \E_2(p)) + X^2\E_2(q)(\E_2(p) -2 u_\alpha)\Bigr),\\\\
   P(X)=\E_2(p) ( 3 u_\alpha -2  \E_2(q) )+3X^2 \E_2(p)(2\E_2(p) -3 u_\alpha)+ 2 X^3 \E_2(q) (3u_\alpha-2  \E_2(p) ).
\end{array}  
\right.
\end{equation*}
From \eqref{rhsine2} and Proposition \ref{lemmainvarianti} we infer that $h(\Lambda(p,q),p,q)$ is positive. Therefore 
in order to prove \eqref{ine1} we only need to consider the case when $h(\lambda(p,q),p,q)$ is positive, which implies, by  
\eqref{hoggi}, that $\frac{P(X)}{Q(X)}$ is positive too.  
In order to proceed we need to show that $Q>0$. We prove, equivalently, that $P>0$. Using  the variable $\lambda$ we note that 
$P(X)>0$ if and only if 
 $$
 \tilde{p}(\lambda):=\E^2_2(p)  (3 u_\alpha-2 \E_2(p))(2\lambda^3 -3\lambda^2) + 
    \E_2^2(q) (2 \E_2(q) - 3 u_\alpha)>0.$$
It is easy to check that for $\lambda>0$, $\tilde{p}$   has a global minimum at $\lambda=1$ and that $\tilde{p}(1)>0$, for all admissible  $p,q$.
Next, we note that
 that \eqref{ine1} holds if and only if
 $$ \hbox{$3Q(X)-P(X)\geq 0$ for any $\lambda\in A(M_{\alpha}(p), M_{\alpha}(q))$}.$$
We compute $3Q(X)-P(X)$ and obtain
\begin{equation*}
      3Q(X)-P(X)= (1-X)^2 \E_2(p) (2 \E_2(q) +X \E_2(q)-3u_\alpha ),
\end{equation*}
which shows that $X=1$ is in fact a double root and not merely a simple root.
Therefore  $3Q(X)-P(X)\geq 0 $ if and only if for each $0\,\leq q\leq p\leq 1,\, p\neq1$ and
$\lambda\in A(M_{\alpha}(p), M_{\alpha}(q))$
\begin{equation}\label{nonnegativa}
\begin{array}{l}
2 \E_2(q)+  \lambda\E_2(p)-3 u_\alpha\geq 0.
    \end{array}
\end{equation}
We now recall that $\min\{\lambda\in A(M_{\alpha}(p), M_{\alpha}(q))\}=\alpha_-(M_{\alpha}(p), M_{\alpha}(q))$, hence \eqref{nonnegativa} holds if and only if for each $0\,\leq q\leq p\leq 1,\, p\neq1$
\begin{equation*}     
2 \E_2(q)+ \alpha_-(M_{\alpha}(p),M_{\alpha}(q)) \E_2(p)-3u_\alpha\geq 0,
\end{equation*}
which  is in turn equivalent to
\begin{equation}
\label{minj>0}
       \alpha_-(M_{\alpha}(p),M_{\alpha}(q))+\frac{2\E_2(q)}{\E_2(p)}-\frac{3u_\alpha}{\E_2(p)}\geq 0.
\end{equation}
We have thus proved that \eqref{ine1}  is equivalent to \eqref{minj>0}, which we proceed to prove.  
For $p\neq 1$, the matrix $M_\alpha(p)$ has distinct eigenvalues, hence Lemma \ref{LemmaF} applies to the pair $(M_\alpha(p),M_\alpha(q))$.
By the definition of $\alpha_-(M_{\alpha}(p),M_{\alpha}(q))$, (see \eqref{ameno}) we need to prove that
\begin{equation*}%
\max\left\{\frac{\E_1(q)}{\E_2(p)}+\frac{2\E_2(q)}{\E_2(p)}- \frac{3u_{\alpha}}{\E_2(p)},\frac{\E_2(q)}{\E_3(p)}+\frac{2\E_2(q)}{\E_2(p)}- \frac{3u_{\alpha}}{\E_2(p)}\right\}\geq0.
\end{equation*}
It suffices to prove that the leftmost expression above is nonnegative, namely that 
\begin{equation}\label{ccomb}
    \begin{array}{ll}
K(q):=\E_1(q) + 2\E_2(q)-3 u_{\alpha}\geq 0,\, \forall q\in[0,1],
     \end{array}
\end{equation}
since $\E_2(p)$ is always strictly positive.
We now regard $K$ as $K(q)=k(\E_2(q))$ with
$k(t)=\E_1(t) + 2 t -3 u_\alpha$. It is well known  that
 the function $t\in[0,1]\to \E_1(t)$  is concave, hence, so is  $k$. Therefore $k$
 is minimised at the boundary of its domain of definition. Since the function $\E_2$ varies monotonically in $[u_{\alpha},s_2]$ as $q$ ranges in $[0,1]$, we find
$$\min_{q\in[0,1]}K(q)=\min_{\E_2\in[u_{\alpha},s_2]}k(\E_2)=\min\{k(u_{\alpha}), k(s_2)\}.$$
We now verify that $\min\{k(u_{\alpha}), k(s_2)\}=0$.
Note that
 $$ \min\{k(u_{\alpha}), k(s_2)\}=\min\{u_{\alpha}+2u_{\alpha}-3u_{\alpha}, s_1+2s_2-3u_{\alpha}\}=\min\{0,s_1+2s_2-3u_{\alpha}\}.$$
It remains to show that  $s_1+2s_2-3u_{\alpha}\geq 0$. Recalling \eqref{rootU}, one checks that
$$b\left(\frac{s_1+2 s_2}{3}\right)=\frac{s_1(s_2-s_1)(s_2-s_3)}{3}<0.$$ Hence, the value $\frac{s_1+2s_2}{3}$  lies between the two roots of $b=0$. 
In particular, it is greater than the smaller root, namely $u_{\alpha}$, which is positive.  Thus
$\displaystyle{\min_{\E_2\in[u_{\alpha},s_2]}k(\E_2)=0}.$ Hence \eqref{ine1} holds.

\end{proof}

\section{Computation of the invariants}
We now establish the technical lemma used in the proof of Proposition \ref{xy} and Proposition \ref{lemmainvarianti}.
  Recall \eqref{s11}, \eqref{xF}, \eqref{6.15} and define 
\begin{equation}
\label{i2i3}
x(t):=x(\alpha,t,S,U_\alpha)=i_2(M_{\alpha}(t)),\quad y(t):=y(\alpha,t,S,U_\alpha)=i_3(M_{\alpha}(t)).
\end{equation}

\begin{Lemma}\label{usid}
Under the assumptions of Proposition \ref{xy},  we have $\lambda\notin \{0,1\}$ and, for each $t\in[0,1]$, one has
\begin{equation}\label{s92}
\left\{
\begin{array}{ll}
x(\lambda, t, F,G)=\displaystyle{
\frac{x_G-\lambda^2 x_F}{\lambda(1-\lambda)}\,\eta(\lambda,t)-\frac{x_G-\lambda x_F}{\lambda(1-\lambda)}\eta^2(\lambda,t)\,}\,\\\\
y(\lambda,t,F,G)=\displaystyle{\frac{y_G-\lambda^3 y_F}{\lambda^2(1-\lambda)}\,\eta^2(\lambda,t)
-\frac{y_G-\lambda^2 y_F}{\lambda^2(1-\lambda)}\,\eta^3(\lambda,t).}
\end{array}
\right.
\end{equation}
In particular, 

\begin{equation}\label{mainxy}
\begin{array}{lll}
\left\{\begin{array}{l}
x(t)
=\displaystyle{
 \frac{2-3u_\alpha}{u_\alpha
 } \left(2u_\alpha -\mathcal E_2(t)
\right)}\E_2(t)
\\\\
y(t)
=\displaystyle{
\frac{1-2 u_\alpha}{u_\alpha}\left(
3u_\alpha\, -2 \mathcal E_2(t)\right)\mathcal E_2^2(t)\,} ,
\end{array}
\right.
\end{array}
\end{equation}
and, for any  $0\leq p\leq q\leq 1$, with $p\neq 1$, we have  
\begin{equation}\label{mainxyz}
\begin{array}{cc}
\displaystyle{\frac{x(p)}{x(q)}=r(p,q)},\quad&
\displaystyle{\frac{y(p)}{y(q)}=s(p,q)}.
\end{array} 
\end{equation}
\end{Lemma}
\begin{proof}
We first prove \eqref{s92}. 
For $(\xi_1,\xi_2,\xi_3)\in \R^3$, with $\xi_1\notin\{0,1\}$, define
\begin{equation*}
\begin{array}{lll}
\displaystyle{
w_1(\xi_1,\xi_2,\xi_3):=
\frac{\xi_3- \xi_1 \xi_2}{1-\xi_1
}
},
&
\displaystyle{
w_2(\xi_1,\xi_2,\xi_3):=
\frac{\xi_3- \xi_1^2 \xi_2}{\xi_1(1-\xi_1)
}
}
,&
\displaystyle{
w_3(\xi_1,\xi_2,\xi_3):=
\frac{\xi_3- \xi_1^3 \xi_2}{\xi_1^2(1-\xi_1)
}
}.
\end{array}
\end{equation*}
We begin by showing that the functions $w_i$ are constant along the curves defined by the three homogeneous symmetric invariants $i_j$. 
In other words,
for each $t\in [0,1]$, we prove that
\begin{equation*}
    \left\{\begin{array}{l}
\displaystyle{
w_1(\eta(\lambda,t),i_1(F), i_1(G))=w_1(\lambda,i_1(F), i_1(G))
},
\\\\
\displaystyle{
w_2(\eta(\lambda,t),x_F,x_G)=w_2(\lambda,x_F,x_G)
},\quad
\displaystyle{
w_3(\eta(\lambda,t),y_F,y_G)=w_3(\lambda,y_F,y_G)
}.
\end{array}
\right.
\end{equation*}
Indeed, denoting by $f_j$ the eigenvalues of $F$,  we have
\begin{equation}\label{cpnew}
\begin{array}{cc}
\det (M_{F,G}(\lambda,t) -zI)\\
\\
=\displaystyle{
\det(\eta(\lambda,t) F - z I)+(1-\eta(\lambda,t))\sum_{i=1}^{3} n_i^2(\eta(\lambda,t) f_j - z)(\eta(\lambda,t) f_k - z).
}
\end{array}\end{equation}
The last sum is taken with the convention that $i\neq j\neq k\neq i$.
As a function of the complex variable $z$, the equation \eqref{cpnew} says that a certain polynomial of degree three vanishes identically. Therefore, the coefficient of the polynomial on the left-hand side must be equal to those on the right-hand side. The coefficient of degree three is equal to  -1 for both polynomials. We write the other three equalities.
We have 
\begin{equation}\label{coeff0}
    \begin{array}{ll}
    i_1(M(\lambda,t)) =\eta(\lambda,t) i_1(F)+(1- \eta(\lambda,t))[n_1^2+n_2^2+n_3^2]\\\\
  i_2(M(\lambda,t)) =    \eta^2(\lambda,t) x_F+ \eta(\lambda,t)(1- \eta(\lambda,t))[n_1^2(f_2+f_3)+n_2^2(f_3+f_1)+n_3^2(f_1+f_2)] \\\\
   i_3(M(\lambda,t))=    \eta^3(\lambda,t) y_F+ \eta^2(\lambda,t)(1- \eta(\lambda,t))   [n_1^2f_2f_3+n_2^2f_3f_1+n_3^2f_1f_2].
    \end{array}
\end{equation}
Isolating the terms depending explicitly on the $n_i$'s one can see that the functions $w_j$ are independent of $t$. Therefore evaluating $w_j$ at $t=0$ we get 
\begin{equation}\label{coeff}
\left\{
    \begin{array}{lll}        
    n_1^2+n_2^2+n_3^2=
    \displaystyle{
\frac{i_1(G)-\lambda i_1(F)}{1- \lambda}}, \\
n_1^2(f_2+f_3)+n_2^2(f_3+f_1)+n_3^2(f_1+f_2)=\displaystyle{\frac{x_G-\lambda^2 x_F}{\lambda(1- \lambda)}},   \\
  n_1^2f_2f_3+n_2^2f_3f_1+n_3^2f_1f_2=\displaystyle{
      \frac{y_G-\lambda^3 y_F}{\lambda^2(1- \lambda)}}.
\end{array}
    \right.
\end{equation}
Note that with our assumptions on $F$ and $G$ one has $\lambda \notin\{0,1\}$, so the formulas in \eqref{coeff} make sense. Moreover, our assumptions imply that $i_1(M(\lambda,t))=i_1(F)=1$, thus the first equation is automatically satisfied since $n$ has norm one.
Plugging \eqref{coeff0} into \eqref{coeff} we obtain 
\begin{equation*}
x(\lambda, t, F,G)= \eta^2(\lambda,t) x_F+ \eta(\lambda,t)\bigl(1-\eta(\lambda,t)\bigr)w_2(\lambda,x_F,x_G) 
\end{equation*} 
which yields the first equation in \eqref{s92}. 
The derivation of the second one is similar and is thus omitted.
\par

It remains to show \eqref{mainxy} and \eqref{mainxyz}. The latter follows from the former by very elementary algebraic manipulations which we omit. 
We focus on the first formula of \eqref{mainxy}.
Specialising \eqref{s92} to $F=S$, $G=U_\alpha$, $\lambda=\alpha$ and recalling that $\E_2(t) = s_2\eta(t)$, 
we find 
\begin{equation}
    \begin{array}{c}
    \displaystyle{
    x(\alpha,t,S,U_\alpha)
    =
\frac{x_{U_\alpha}-\alpha^2 x_S}{s_2\alpha(1-\alpha)}\,\E_2(t)-\frac{x_{U_\alpha}-\alpha x_S}{s_2^2\alpha(1-\alpha)}\E_2^2(t)}.
    \end{array}
\end{equation}
Therefore, it is enough to prove that
\begin{equation}\label{temp100}
    \displaystyle{
   \frac{x_{U_\alpha}-\alpha^2 x_S}{s_2\alpha(1-\alpha)}=2(2-3u_\alpha), 
\quad 
\frac{x_{U_\alpha}-\alpha x_S}{s_2^2\alpha(1-\alpha)}=\frac{2 -3u_\alpha}{u_\alpha}.
   }
\end{equation}
It is easy to check that the two equalities  in \eqref{temp100} are equivalent to following equation
\begin{equation}\label{temp101}
   6s_2 u^2_\alpha+(x_S-4s_2-3s_2^2)u_\alpha+2s_2^2=0.
\end{equation}
We now show that \eqref{temp101} is equivalent to $b(x)=0$, with $b$ defined in  \eqref{rootU}, which holds by construction.
Indeed,
$$
x_S-4s_2-3s_2^2=s_1s_3+s_2-s_2^2-4s_2-3s_2^2=s_1s_3-3s_2-4s_2^2.
$$
The proof of the second equation in \eqref{mainxy} is similar and is omitted.
\end{proof}
\section{The $\gamma_{\beta}$-case and the proof of Lemma \ref{davveroultimo}}\label{final7} 

\subsection{The $\gamma_{\beta}$-case}
We will not treat the present case in detail, since the arguments are parallel to those used for the $\gamma_\alpha$-case.
Recall Definitions \ref{def_U} and \ref{defgamma} and  denote the ordered eigenvalues of $M_{\beta}(t)$ by $\F_1(t)\leq \F_2(t)\leq\F_3(t)$.
Set  
  $$
   \Lambda_\beta(p,q):=
    \frac{\F_2(q)}{\F_2(p)}, \qquad r_\beta(p,q):=\frac{\F_2(p)}{\F_2(q)}\left(\frac{2u_\beta - \F_2(p)}{2u_\beta - \F_2(q)}\right). 
   $$
The following result is the analog of \eqref{crucial} and  of Proposition \ref{propa}.
   
\begin{proposition}\label{propb}
  Let  $0\leq p \leq q \leq 1$, $p\neq 1$. 
  There exist $R_p, R_q \in SO(3)$ such that 
 \begin{equation*}
    \begin{array}{cc}
       R_q^t M_{\beta}(q) R_q = \Lambda_\beta(p,q) R_p^t M_{\beta}(p) R_p+(1-\Lambda_\beta(p,q))n_\beta\otimes n_\beta,
    \end{array}
\end{equation*}   
 and
 \begin{equation}\label{conj1b}
 \min_{\lambda \in A(M_{\beta}(p),M_{\beta}(q))} h(\lambda,p,q)= h(\Lambda_\beta(p,q),p,q).
 \end{equation}   
\end{proposition}
The proof of Proposition \ref{propb} is based on arguments that are very similar to those used to prove Proposition \ref{propa} and is therefore omitted. We only point out that the minimisation of 
$h$ in \eqref{conj1b} is justified by Lemma \ref{davveroultimo} (see Remark \ref{remutile}).

For completeness, in the next lemma we summarise the 
analog of the results of Proposition \ref{lemmainvarianti} and Lemma \ref{83}.
\begin{Lemma}\label{c-lemmab}
 For each $0\leq p \leq q\leq 1$, $p\neq 1$, and $\lambda\in A(M_{\beta}(p),M_{\beta}(q))$, we have 
 \begin{equation}\label{come4.14}
3u_\beta -2 \F_2(q) >0, \quad
1-\lambda(2-\lambda)r_\beta(p,q) >0,
\end{equation}
 \begin{equation}\label{amenob}
 \beta_-(M_\beta(p),M_\beta(q))=\Lambda_\beta(p,q),\quad
 {\displaystyle  \beta_+(M_\beta(p),M_\beta(q))=
 \max\left\{\frac{\F_2(q)}{\F_1(p)},\frac{ {\F_3(q)}}{\F_2(p)}\right\}},
 \end{equation}
and
\begin{equation}\label{973}
 q<1 \Longrightarrow  \beta_-(M_\beta(p),M_\beta(q))<\beta_+(M_\beta(p),M_\beta(q)).
\end{equation}
\end{Lemma}

\subsection{Proof of Lemma \ref{davveroultimo}}\label{posticipato} 
We add the subscript $\alpha$ in the definition \eqref{i2i3}. Similarly we define 
\begin{equation*}
x_\beta(t):=x(\beta,t,S,U_\beta)=i_2(M_{\beta}(t)),\quad y_\beta(t):=y(\beta,t,S,U_\beta)=i_3(M_{\beta}(t)).
\end{equation*}
We use \eqref{mainxy} to find
\begin{equation*}
\dot{x}_\alpha(t)=
2(2-3 u_\alpha)\Big(1-\frac{\E_2(t)}{u_\alpha}\Big)\dot\E_2(t)
,\quad
\dot{y}_\alpha(t)=
6 (1-2 u_\alpha)
\Big(1-\frac{\E_2(t)}{u_\alpha}\Big)
\E_2(t) \dot\E_2(t).
\end{equation*} 
Similarly, from \eqref{s92}, we obtain
$$
\dot{x}_\beta(t)=
2(2-3 u_\beta)\Big(1-\frac{\F_2(t)}{u_\beta}\Big)\dot\F_2(t)
,\quad
\dot{y}_\beta(t)=
6 (1-2 u_\beta)
\Big(1-\frac{\F_2(t)}{u_\beta}\Big)
\F_2(t) \dot\F_2(t).
$$
Taking into account \eqref{wd} and \eqref{come4.14}, one can see that both $\dot{x}_\alpha(t)$ and $\dot{x}_\beta(t)$  are strictly positive for each $t\in (0,1)$. 
Moreover, since $\E_2$ is strictly decreasing while $\F_2$ is strictly increasing,  $\gamma_\alpha$ and $\gamma_\beta$ intersect only at the point $(i_2(S),i_3(S))$.
We now compute the slope of the tangent vectors at $(i_2(S),i_3(S))$ and find
$$
\frac{\dot y_\alpha(0)}{\dot x_\alpha(0)}=\frac{3 s_2(1-2 u_\alpha)}{2-3 u_\alpha},\quad \frac{\dot y_\beta(0)}{\dot x_\beta(0)}= \frac{3 s_2(1-2 u_\beta)}{2-3 u_\beta}.
$$ 
Finally
$$
\frac{1-2 u_\alpha}{2-3 u_\alpha}>\frac{1-2 u_\beta}{2-3 u_\beta} \Longleftrightarrow u_\alpha<u_\beta, 
$$  
which follows from (i) of Proposition \ref{prop_opt_curve}.

\section*{Acknowledgements} This material is based upon work supported by the National Science Foundation under Grant No.~2108588 (N.~Albin, Co-PI). V. Nesi gratefully acknowledges  Progetti di Ricerca di Ateneo 2022, ``Equazioni ellittiche e paraboliche non lineari'', rif. RM1221816BBB81FA e Progetto di Ateneo 2023, rif. RG123188B475BBF0, titolo ``Singularities and interfaces: a variational approach to stationary and evolution problems.''
M. Palombaro is a member of the Gruppo Nazionale per l'Analisi Matematica, la Probabilit\`a e le loro applicazioni (GNAMPA) of the Istituto Nazionale di Alta Matematica (INdAM). 
\section*{Declarations}
The authors have no conflicts of interest. Data sharing is not applicable to this article as no datasets were generated or analysed during the current study.

\end{document}